\documentclass[12pt]{amsart}
\usepackage{amsmath}
\usepackage{amssymb}
\usepackage{amsfonts}
\usepackage{latexsym}
\usepackage{amscd}
\usepackage[mathscr]{euscript}
\usepackage{enumitem}
\usepackage{xy} \xyoption{all}
\usepackage{pdfsync}
\usepackage[active]{srcltx}

\newcommand{\N}{{\mathbb{N}}}

\newcommand{\Z}{{\mathbb{Z}}}

\newcommand {\OGE }{{\mathcal {O}_{G,E}}}

\newcommand{\CG}{{\mathcal{G}_{\text{tight}}^{(G,E)}}}
\newcommand {\SGE }{{\mathcal {S}_{G,E}}}

\newcommand{\Eu}{\widehat E_\infty}
\newcommand{\Et}{\widehat E_{\text{tight}}}

\newtheorem{lemma}{Lemma}[section]
\newtheorem{corollary}[lemma]{Corollary}
\newtheorem{theorem}[lemma]{Theorem}
\newtheorem{proposition}[lemma]{Proposition}
\newtheorem{remark}[lemma]{Remark}
\newtheorem{definition}[lemma]{Definition}

\newtheorem{noname}[lemma]{}

   \usepackage{color}

   \usepackage{color}

\begin{document}

\title[Arbitrary self-similar graphs]{$C^*$-algebras of self-similar graphs over arbitrary graphs}%

\author {Ruy Exel}
\address {Ruy Exel\\Departamento de Matem\'atica, Universidade Federal de Santa Catarina, 88040-970 Florian\'opolis SC, Brazil}
\email {exel@mtm.ufsc.br}\urladdr {http://www.mtm.ufsc.br/~exel/}

\author {Enrique Pardo}
\address {Enrique Pardo\\Departamento de Matem\'aticas, Facultad de Ciencias\\ Universidad de C\'adiz, Campus de
Puerto Real\\ 11510 Puerto Real (C\'adiz)\\ Spain.}
\email {enrique.pardo@uca.es}\urladdr {https://sites.google.com/a/gm.uca.es/enrique-pardo-s-home-page/}

\author{Charles Starling}
\address{Charles Starling\\Carleton University, School of Mathematics and Statistics, 4302 Herzberg Laboratories, 1125 Colonel By Drive, Ottawa, ON, Canada, K1S 5B6.}
\email{cstar@math.carleton.ca}\urladdr{https://carleton.ca/math/people/charles-starling/}


\thanks{The first-named author was partially supported by CNPq. The third named author was partially supported by PAI III grant FQM-298 of the
Junta de Andaluc\'{\i }a, and by the DGI-MINECO and European Regional Development Fund, jointly, through grants MTM2014-53644-P and MTM2017-83487-P. The third named author was partially supported by a Carleton University internal research grant.}




\begin{abstract}
In this note we extend the construction of a $C^*$-algebra associated to a self-similar graph to the case of arbitrary countable graphs. We reduce the problem to the row-finite case with no sources, by using a desingularization process. Finally, we characterize simplicity in this case.
\end{abstract}
\maketitle

\section* {Introduction}

$C^*$-algebras of self-similar graphs were introduced by the first and second authors in \cite{EP1} for the case of discrete countable groups and finite graphs with no sources. Lately, B\'edos, Kaliszewski and Quigg \cite{BKQ} extended the definition to arbitrary groups and (topological) graphs, using a Cuntz-Pimsner picture of these algebras; unfortunately, they have not provided characterizations of simple or purely infinite for these algebras under this point of view. Also, no one has developed \underline{directly} the results in \cite{EP1} to the context of countable discrete groups and arbitrary graphs. This note we fill the gap, by reducing the problem row-finite graphs with no sources, and then showing how characterizations stated in \cite{EP1} works correctly under this restrictions. The results in this note are essential to extend the scope of the results obtained in \cite{EP1} for Katsura algebras over finite matrices to the general case; this guarantees, by \cite{Kat1}, that every Kirchberg algebra in the UCT is the full groupoid $C^*$-algebra of a second countable amenable ample groupoid.

\section{The case of finite graphs}\label{sec:ssE}

In this section we will recall the essential items needed to understand the algebras $\OGE$ associated to triples $(G,E, \varphi)$, introduced in \cite{EP1}. Let us recall the construction.

\begin{noname}\label{basicdata}
{\rm The basic data for our construction is a triple $(G,E,\varphi)$ composed of:
\begin{enumerate}
\item A finite directed graph $E=(E^0, E^1, r, s)$ without sources.
\item A discrete group $G$ acting on $E$ by graph automorphisms.
\item A 1-cocycle $\varphi: G\times E^1\rightarrow G$ satisfying the property
$$
\varphi(g,a)\cdot x=g\cdot x \text{ for every } g\in G, a\in E^1, x\in E^0.
$$
\end{enumerate}
}
\end{noname}
The property $(3)$ required on $\varphi$ is tagged $(2.3)$ in \cite{EP1}. As remarked in \cite{BKQ}, this condition can be weakened to
$$
\varphi(g,a)\cdot s(a)=g\cdot s(a) \text{ for every } g\in G, a\in E^1,
$$
without pain.

\begin{definition}\label{Def:OGE}
{\rm Given a triple $(G,E, \varphi)$ as in  (\ref{basicdata}), we define $\OGE$ to be the universal $C^*$-algebra as follows:
\begin{enumerate}
\item \underline{Generators}:
$$\{p_x : x\in E^0\}\cup\{s_a : a\in E^1\} \cup \{u_g : g\in G\}.$$
\item \underline{Relations}:
\begin{enumerate}
\item $\{p_x : x\in E^0\}\cup\{s_a : a\in E^1\}$ is a Cuntz-Krieger $E$-family in the sense of \cite{Raeburn}.
\item The map $u:G\rightarrow \OGE$ defined by the rule $g\mapsto u_g$ is a unitary $\ast$-representation of $G$.
\item $u_gs_a=s_{g\cdot a}u_{\varphi(g,a)}$ for every $g\in G, a\in E^1$.
\item $u_gp_x=p_{g\cdot x}u_g$ for every $g\in G, x\in E^0$.
\end{enumerate}
\end{enumerate}
}
\end{definition}
\noindent Notice that the relation (2a) in Definition \ref{Def:OGE} implies that there is a natural representation map
$$
\begin{array}{cccc}
\phi: & C^*(E) &\to   & \OGE  \\
 & p_x & \mapsto  & p_x \\
 & s_a & \mapsto  & s_a  
\end{array}
$$
which is injective \cite[Proposition 11.1]{EP1}.

Recall from \cite[Definition 4.1]{EP1} that given a triple $(G,E,\varphi)$ as in (\ref{basicdata}), we define an inverse semigroup $\SGE$ as follows:
\begin{enumerate}
\item The set is
$$\SGE=\{ (\alpha,g,\beta) : \alpha, \beta\in E^*, g\in G, s(\alpha)=gs(\beta)\}\cup \{ 0\}\},$$
where $E^*$ denotes the set of finite paths in $E$.
\item The operation is defined by:
$$(\alpha,g,\beta)\cdot (\gamma,h,\delta):=
\left\{
\begin{array}{cc}
 (\alpha, g\varphi(h,\varepsilon), \delta h\varepsilon), &   \text{if } \beta=\gamma \varepsilon  \\
   &  \\
 (\alpha g\varepsilon, \varphi(g,\varepsilon) h, \delta), &  \text{if } \gamma=\beta\varepsilon    \\
    &  \\
 0, & \text{otherwise,}
\end{array}
\right.
$$
and $(\alpha,g,\beta)^*:= (\beta,g^{-1}, \alpha)$.
\end{enumerate}

Then, we can construct the groupoid of germs of the action of $\SGE$ on the space of tight filters $\Et(\SGE)$
of the semilattice ${E}(\SGE)$ of idempotents of $\SGE$. In our concrete case, $\Et(\SGE)$ turns out to be 
 homeomorphic to the compact space $E^{\infty}$ of one-sided infinite paths on $E$; in particular, $\Et(\SGE)=\Eu(\SGE)$. Hence, the action of 
 $(\alpha,g,\beta)\in \SGE$ on $\eta=\beta\widehat{\eta}$  is given by the rule $(\alpha,g,\beta)\cdot \eta=\alpha (g\widehat{\eta})$. 
Thus, the groupoid of germs is
$$\CG=\{[\alpha,g,\beta;\eta] : \eta=\beta\widehat{\eta}\},$$
where $[s;\eta]=[t;\mu]$ if and only if $\eta=\mu$ and there exists $0\ne e^2=e\in \SGE$ such that $e\cdot\eta=\eta$ and $se=te$.
The unit space 
\[\CG^{(0)}=\{[\alpha,1,\alpha;\eta] : \eta=\alpha\widehat{\eta}\}\]
is identified with the one-sided infinite path space $E^{\infty}$, via the
homeomorphism $[\alpha,1,\alpha; \eta]   \mapsto   \eta.$
Under this identification, the range and source maps on
$\CG$ are:
\[s([\alpha,g,\beta;\beta \widehat{\eta}])= \beta\widehat{\eta} \quad \text{ and } \quad  
r([\alpha,g,\beta;\beta \widehat{\eta}])= \alpha(g \widehat{\eta}).\]

A basis for the topology on $\CG$ is given by compact open bisections of the form 
\[ \Theta(\alpha,g,\beta;Z(\gamma)):= \{ [\alpha, g, \beta;\xi ] \in \CG : \xi \in  Z(\gamma) \} \] 
where $\gamma \in E^*$ and  $Z(\gamma):=\{\gamma \widehat{\eta} : \widehat{\eta}\in E^{\infty}\}.$
Thus $\CG$ is locally compact and ample.  
In \cite{EP1} characterizations are given for when 
$\CG$ is Hausdorff \cite[Theorem 12.2]{EP1}, 
amenable \cite[Corollary 10.18]{EP1}, minimal \cite[Theorem 13.6]{EP1} or effective \cite[Theorem 14.10]{EP1} 
in terms of the properties of the triple $(G,E,\varphi)$ and the action of $\SGE$ on $E^{\infty}$.


\section{Extending to the countable case}

In this section we will look at the problem of extending our class to the case of countably infinite graphs with no restrictions (i.e. sources, sinks and infinite receivers are admitted).

Notice that, if $E^0$ is countably infinite, then the algebra $C^*(E)$ is not longer unital. Thus, if we pretend to get a unitary representation of the group $G$ associated to the algebra, we need to consider unitary representations of $G$ in the multiplier algebra $\mathcal{M}(\OGE)$. This is not a good idea for giving an intrinsic definition of the object, but it is very helpful to deal with the details in the definition. Nevertheless, the model we will follow is that of Katsura algebras \cite{Kat1}, where the unitary associated to an element of $\Z$ is written in terms of partial unitaries associated to the projections $p_x$ for $x\in E^0$.\vspace{.2truecm}

\begin{noname}\label{basicdata2}
{\rm The basic data for our construction is a triple $(G,E,\varphi)$ composed of:
\begin{enumerate}
\item A countable directed graph $E=(E^0, E^1, r, s)$.
\item A discrete group $G$ acting on $E$ by graph automorphisms.
\item A 1-cocycle $\varphi: G\times E^1\rightarrow G$ satisfying the property
$$
\varphi(g,a)\cdot x=g\cdot x \text{ for every } g\in G, a\in E^1, x\in E^0.
$$
\end{enumerate}
}
\end{noname}
As remarked before, this condition can be weakened to
$$
\varphi(g,a)\cdot s(a)=g\cdot s(a) \text{ for every } g\in G, a\in E^1,
$$
without pain. Notice that all the results in \cite[Section 2]{EP1} are true for triples $(G,E, \varphi)$ as in (\ref{basicdata2}). So, assume that we have a such triple $(G,E, \varphi)$. To any $x\in E^0$ and any $g\in G$ we will associate an element $u_{g,x}\in \OGE$. Since we want that $\sum\limits_{x\in E^0}u_{g,x}$ converges to a unitary $u_g\in \mathcal{M}(\OGE)$ in the strong topology, our idea is to think that $u_{g,x}=u_gp_x$ (recall that $1_{\mathcal{M}(\OGE)}=\sum\limits_{x\in E^0}p_{x}$). Using the relations enjoyed by $u_gp_x$ in the original definition, we conclude that $u_{g,x}$ is a partial isometry with:
\begin{enumerate}
\item $u_{g,x}u_{g,x}^*=p_{g\cdot x}$.
\item $u_{g,x}^*u_{g,x}=p_{x}$.
\item $u_{g, s(a)}s_a=s_{g\cdot s(a)}u_{\varphi (g,a), s(a)}$.
\item $u_{g,x}p_x=p_{g\cdot x}u_{g,x}$.
\end{enumerate}

With this definition, $\sum\limits_{x\in E^0}u_{g,x}$ converges to an element $u_g\in \mathcal{M}(\OGE)$, and it is easy to see that it is a unitary. Since we are interested in getting a unitary representation of $G$ in $\mathcal{M}(\OGE)$, we need to be sure that $u_gu_h=u_{gh}$ for every $g,h\in G$. A simple computation shows that this occurs whenever 
$$u_{gh, h^{-1}\cdot x}=u_{g,x}u_{h, h^{-1}\cdot x}$$
for every $g,h\in G$ and every $x\in E^0$. In view of all that facts, we obtain the following definition

\begin{definition}\label{Def:OGE_infty}
{\rm Given a triple $(G,E, \varphi)$ as in  (\ref{basicdata2}), we define $\OGE$ to be the universal $C^*$-algebra as follows:
\begin{enumerate}
\item \underline{Generators}: 
$$\{p_x\mid x\in E^0\}\cup\{s_a\mid a\in E^1\} \cup \{u_{g,x}\mid g\in G,\, x\in E^0\}.$$
\item \underline{Relations}:
\begin{enumerate}
\item $\{p_x\mid x\in E^0\}\cup\{s_a\mid a\in E^1\}$ is a Cuntz-Krieger $E$-family in the sense of \cite{Raeburn}.
\item $u_{g,x}$ is a partial isometry with:
\begin{enumerate}
\item $u_{g,x}u_{g,x}^*=p_{g\cdot x}$.
\item $u_{g,x}^*u_{g,x}=p_{x}$.
\end{enumerate}
\item $u_{gh, h^{-1}\cdot x}=u_{g,x}u_{h, h^{-1}\cdot x}$ for every $g,h\in G$ and $x\in E^0$.
\item $u_{g, s(a)}s_a=s_{g\cdot a}u_{\varphi(g,a), s(a)}$ for every $g\in G, a\in E^1$.
\item $u_{g,x}p_x=p_{g\cdot x}u_{g, x}$ for every $g\in G, x\in E^0$.
\end{enumerate}
\end{enumerate}
}
\end{definition}

\begin{remark}\label{Rem:OGE_infty}
{\rm
When $E^0$ is finite, Definition \ref{Def:OGE_infty} coincides with Definition \ref{Def:OGE}. Moreover, we have:
\begin{enumerate}
\item $u_g:=\sum\limits_{x\in E^0}u_{g,x}$ is a unitary in $\mathcal{M}(\OGE)$.
\item The map $u:G\rightarrow \mathcal{M}(\OGE)$ defined by the rule $g\mapsto u_g$ is a unitary $\ast$-representation of $G$.
\item $u_{g,x}=u_gp_x$ for every $g\in G$ and $x\in E^0$.
\item $u_gp_x=p_{gx}u_{\varphi(g,a)}$ for every $g\in G$ and $x\in E^0$.
\item $u_gs_a=s_{ga}u_{\varphi(g,a)}$ for every $g\in G$ and $a\in E^1$.
\end{enumerate}
}
\end{remark}

Following the same structure, we will need to associate an abstract inverse semigroup to this algebra. The natural one will be

\begin{definition}\label{Def:inverseSemigroup_infty}
{\rm Given a pair $(G,E)$ as in  (\ref{basicdata}), we define a $\ast$-semigroup $\widehat{\SGE}$ as follows. As a set,
$$\widehat{\SGE}=\{ (\alpha,(g,x),\beta)\mid \alpha,\beta\in E^*, g\in G, g\cdot x=s(\alpha)=g\cdot s(\beta)\}\cup \{ 0\}.$$
The operation is defined as in \cite[Definition 4.1]{EP1}, and the semilattice of idempotents is
$$E(\widehat{\SGE})=\{\alpha , (1,x), \alpha) \mid \alpha \in E^*, s(\alpha)=x\}.$$
}
\end{definition}

Since $x\in E^0$ in the definition must to coincide with $s(\alpha)=g\cdot s(\beta)$, we conclude that $\widehat{\SGE}$ and  $\SGE$ are essentially identical as $\ast$-semigroups, and thus we will switch the notation of the semigroup to $\SGE$. So, with small adaptations, \cite[Sections 3 \& 4]{EP1} hold. Moreover, there is an semigroup homomorphism
$$
\begin{array}{cccc}
 \pi :& \SGE & \longrightarrow   & \OGE  \\
 & (\alpha, (g,x), \beta) & \mapsto  &   s_{\alpha}u_{g,x}s_{\beta}^*
\end{array}.
$$

The essential point is the following result

\begin{theorem}\label{Thm:GroupoidRowFiniteNoSources}
Let $(G,E, \varphi)$ be a triple as in (\ref{basicdata2}) such that $E$ is a row-finite graph with no sources. Then:
\begin{enumerate}
\item The semigroup homomorphism $ \pi :\SGE \rightarrow \OGE $ is a universal tight representation of $\SGE$.
\item $\OGE\cong C^*_{\text{tight}}(\SGE)\cong C^*(\CG)$.
\end{enumerate}
\end{theorem}
\begin{proof}
(1) If $E$ is a row-finite graph with no sources, then the proofs of \cite[Proposition 6.2]{EP1} and \cite[Theorem 6.3]{EP1} works correctly.

(2) Because of (1), we can use the proof of \cite[Theorem 2.4]{Exel2} to conclude the desired result. 
\end{proof}

\vspace{.2truecm}

In order to extend the results in \cite{EP1} to this context, we will need to reduce ourselves to a situation in which the graph $E$ may be assumed to be row-finite without sources. We will show that this is possible via a ``desingularization'' process, inspired in the one developed in \cite{DrinTom} for graph $C^*$-algebras.


\section{Desingularizing triples}\label{sec:desing}

Suppose we have a triple $(G,E, \varphi)$ as in (\ref{basicdata2}), and let $F$ denote the desingularized graph of $E$ obtained in \cite{DrinTom}. In this section we will show that we can define an action $G\curvearrowright F$ and a $1$-cocycle $\widehat{\varphi}:G\times F\rightarrow G$ extending the ones in the original triple such that $\OGE$ and $\mathcal{O}_{G,F}$ are strong Morita equivalent $C^*$-algebras.

\begin{remark}\label{rem:action_sourcesreceivers}
{\rm
Let $(G,E, \varphi)$ be a triple as in (\ref{basicdata2}), and let $x\in E^0$ be a vertex. Then:
\begin{enumerate}
\item If $x$ is a source, then so is $gx$ for every $g\in G$.
\item If $x$ is an infinite receiver, then so is $gx$ for every $g\in G$.
\end{enumerate}
In view of that, when defining the desingularization we need to keep track of the fact that, for any vertex in the orbit o a singular vertex, we must define the tails added to it as orbit-connected parts of the graph.
}
\end{remark}


\subsection{Desingularizing a source}\label{ssec:source}

Now, we will explain how to construct a triple $(G,F, \varphi)$ that desingularize a source in a triple $(G,E, \varphi)$. \vspace{.2truecm}

Let $x\in E^0$ be a source. Then:
\begin{enumerate}
\item We define a tail $\{e_i\}_{i\geq 1}\subset F^1$ so that $s(e_i)=r(e_{i+1})$ for every $i\geq 1$, and $r(e_1)=x$.
\item For any $h\in \mbox{St}_G(x)$, we define $he_i=e_i$ for every $i\geq 1$.
\item If $\widehat{G}$ is a set of representatives of the orbits of $x$ under the action $G\curvearrowright E$, then for each $g\in \widehat{G}$ we define $\{e_{i,g}\}_{i\geq 1}\subset F^1$ so that $s(e_{i,g})=r(e_{i+1,g})$ for every $i\geq 1$, and $r(e_{1,g})=gx$, while $ge_i=e_{i,g}$  for every $i\geq 1$.
\end{enumerate}
What we are doing is applying the Drinen-Tomforde desingularization construction in such a way that is coherent with the action $G\curvearrowright E$. 

Next step is to extend the $1$-cocycle. To do this, for any $g,h\in G$ and any $i\geq 1$ we define $\widehat{\varphi} (h, e_{i,g}):=h$. Once this is done, what we have obtained is:
\begin{enumerate}
\item[(4)] A graph $F$ extending $E$, constructed using the Drinen-Tomforde desingularization process.
\item[(5)] An action $G\curvearrowright F$ extending the original action $G\curvearrowright E$, taking care of Remark \ref{rem:action_sourcesreceivers}.
\item[(6)] A $1$-cocycle $\widehat{\varphi}:G\times F\rightarrow G$ extending ${\varphi}:G\times E\rightarrow G$.
\end{enumerate}

Then, the triple $(G,F, \widehat{\varphi})$ is the desingularization of $(G,E, \varphi)$ on to the source $x$.


\subsection{Desingularizing an infinite receiver}\label{ssec:emiter}

Now, we will explain how to construct a triple $(G,F, \varphi)$ that desingularize an infinite receiver in a triple $(G,E, \varphi)$. \vspace{.2truecm}

Let $x\in E^0$ be an infinite receiver, and list $r^{-1}(x)=\{a_i\}_{i\geq 1}$. Then:
\begin{enumerate}
\item We define a tail $\{e_1\}_{i\geq 1}\subset F^1$ so that $s(e_i)=r(e_{i+1})$ for every $i\geq 1$, and $r(e_1)=x$.
\item We define $v_0=x$ and for any $i\geq 1$, $v_i=s(e_i)$.
\item For each $a_j\in r^{-1}(x)$, we define an edge $f_j\in F^1$ from $s(a_j)$ to $v_{j-1}=s(e_j)$.
\item We remove the edges $\{a_i\}_{i\geq 1}$ from $F$.
\item For each $j\geq 1$ we define paths $\alpha_j:=e_1e_2\dots e_{j-1}f_j$.
\item For any $h\in \mbox{St}_G(x)$, we define $he_i=e_i$ and $hf_i=f_i$ for every $i\geq 1$.
\item If $\widehat{G}$ is a set of representatives of the orbits of $x$ under the action $G\curvearrowright E$, then for each $g\in \widehat{G}$ we define $\{e_{i,g}\}_{i\geq 1}\cup \{f_{i,g}\}_{i\geq 1}\subset F^1$ so that $s(e_{i,g})=r(e_{i+1,g})$, $r(f_{i,g})=gr(e_i)$ and $s(f_{i,g})=gs(a_i)$ for every $i\geq 1$, $r(e_{1,g})=gx$, while $ge_i=e_{i,g}$ and $gf_i=f_{i,g}$  for every $i\geq 1$.
\end{enumerate}
What we are doing is applying the Drinen-Tomforde desingularization construction in such a way that is coherent with the action $G\curvearrowright E$. 

Next step is to extend the $1$-cocycle. What we do is, to any $g,h\in G$ and any $i\geq 1$, we define:
\begin{enumerate}
\item[(7)] $\widehat{\varphi} (h, e_{i,g}):=h$. 
\item[(8)] $\widehat{\varphi} (h, f_{i,g})=\varphi(h, a_j)$.
\end{enumerate}
In particular, $\widehat{\varphi} (h, \alpha_{i,g})=\varphi(h, a_j)$. Once this is done, what we have obtained is:
\begin{enumerate}
\item[(9)] A graph $F$ extending $E$, constructed using the Drinen-Tomforde desingularization process.
\item[(10)] An action $G\curvearrowright F$ extending the original action $G\curvearrowright E$, taking care of Remark \ref{rem:action_sourcesreceivers}.
\item[(11)] A $1$-cocycle $\widehat{\varphi}:G\times F\rightarrow G$ extending ${\varphi}:G\times E\rightarrow G$.
\end{enumerate}

Then, the triple $(G,F, \widehat{\varphi})$ is the desingularization of $(G,E, \varphi)$ on to the infinite receiver $x$.


\subsection{The desingularization result}\label{ssec:DesingThm}

Now, we will check that the results of Drinen and Tomforde about their desingularization process for $C^*(E)$ extend to this context.

A simple inspection shows that \cite[Lemmas 2.9 \& 2.10]{DrinTom} extend to our context. Now, we will arrange the proof of \cite[Theorem 2.11]{DrinTom} in order to obtain the desired Morita equivalence between $\OGE$ and $\mathcal{O}_{G,F}$. We will follow the notation of the proof of \cite[Theorem 2.11]{DrinTom}. Let $E$ be a graph with a singular vertex $v_0$, let 
$$\{t_e,q_v \mid e\in F^1, v\in F^0\}$$
be the canonical set of generators for $C^*(F)$, and let
$$\{s_e,p_v \mid e\in E^1, v\in E^0\}$$
be the Cuntz-Krieger $E$-family constructed \textbf{into} $C^*(F)$ in \cite[Lemma 2.9]{DrinTom}. Recall that $\{s_e,p_v\}$ is defined as follows:
\begin{enumerate}
\item For every $v\in E^0$, $p_v:=q_v$.
\item For every $e\in E^1$ such that $r(e)\in E^0_{\text{rg}}$, $s_e:=t_e$.
\item For every $e\in E^1$ such that $r(e)\not\in E^0_{\text{rg}}$, we have that $e=a_j$ for some $j\geq 1$, and thus $s_e:=t_{\alpha_j}$.
\end{enumerate}

Define $B:=C^*(\{s_e,p_v\})$ and $p:=\sum\limits_{v\in E^0}q_v\in \mathcal{M}(C^*(F))$. Then, \cite[Theorem 2.11]{DrinTom} shows that
$$C^*(E)\cong B\cong pC^*(F)p$$
and that $p\in \mathcal{M}(C^*(F))$ is a full projection.

Now, observe that:
\begin{enumerate}
\item For every $g\in G$ and $x\in E^0$ we have that $u_gp_x=u_gq_x=q_{gx}u_g=p_{gx}u_g$.
\item If $r(e)\in E^0_{\text{rg}}$, then $u_gs_e=u_gt_e=t_{ge}u_{\widehat{\varphi}(g,e)}=s_{ge}u_{{\varphi}(g,e)}$, since $\varphi$ and $\widehat{\varphi}$ matches on $G\times E$.
\item If $r(e)\not\in E^0_{\text{rg}}$, then
$$u_gs_e=u_gt_{\alpha_j}=u_gt_{e_1}t_{e_2}\cdots t_{e_{j-1}}t_{f_j}=t_{ge_1}t_{ge_2}\cdots t_{ge_{j-1}}u_gt_{f_j}=$$
$$t_{ge_1}t_{ge_2}\cdots t_{ge_{j-1}}t_{gf_j}u_{\widehat{\varphi}(g, f_j)}=t_{g\alpha_j}u_{\widehat{\varphi}(g, f_j)}=s_{ge}u_{\varphi(g,e)}$$ 
by definition of $\widehat{\varphi}$ in this case.
\end{enumerate}

Thus, the $C^*$-algebra isomorphism
$$
\begin{array}{cccc}
\Phi: & C^*(E) & \rightarrow   & C^*(\{s_e, q_x\})  \\
 & P_x & \mapsto  &  q_x \\
 & T_e & \mapsto  &   s_e
\end{array}
$$
satisfies that:
\begin{enumerate}
\item For every $g\in G$ and $x\in E^0$, $u_g\Phi(P_x)=\Phi(P_{gx})u_g$.
\item For every $g\in G$ and $e\in E^1$, $u_g\Phi(T_e)=\Phi(T_{ge})u_{\widehat{\varphi}(g,e)}$.
\end{enumerate}

Hence, by the universal property of $\OGE$, $Phi$ extends to a $C^*$-algebra isomorphism
$$\Phi: \OGE \rightarrow  C^*(\{u_g, s_e, q_x\})\subseteq \mathcal{O}_{G,F}.$$
Moreover, following \cite[Theorem 2.11]{DrinTom}, we have a $C^*$-isomorphism
$$
\begin{array}{cccc}
\Psi: & \Phi(C^*(E)) & \rightarrow   & pC^*(F)p  \\
 & s_{\alpha}s_{\beta}^* & \mapsto  &  ps_{\alpha}s_{\beta}^*p \\
\end{array}.
$$
Notice that in $pC^*(F)p$, for every $g\in G$ and $x\in E^0$ we have that $pu_gp_xp=pp_{gx}u_gp$, and since $u_gp=pu_g$ by definition of $p$, this is equal to $pp_{gx}pu_g$. Since a homomorphism from $\OGE$ to $p\mathcal{O}_{G,F}p$ extending $\Psi$ should send $u_gp_x\mapsto pu_gpxp$ and $p_{gx}u_g \mapsto pp_{gx}pu_g$, a such extension will be compatible with the defining relations of $\OGE$. Similarly for every $g\in G$ and $a\in E^1$ we have that $u_gs_a$ and $s_{ga}u_{\widehat{\varphi}(g,a)}$ will map to the same element in $p\mathcal{O}_{G,F}p$. So, again by the universal property of $\OGE$, the isomorphism $\Psi$ extends to a $C^*$-algebra isomorphism $\OGE\cong p\mathcal{O}_{G,F}p$. Also, as in \cite[Theorem 2.11]{DrinTom}, $p\in \mathcal{M}(C^*(F))\subset \mathcal{M}(\mathcal{O}_{G,F})$ is a full projection.

Summarizing

\begin{theorem}\label{Thm:Desin}
Let $(G,E, \varphi)$ be a triple, and let $(G,F, \widehat{\varphi})$ be its desingularization. Then, there exists a full projection $p\in \mathcal{M}(\mathcal{O}_{G,F})$ such that $\OGE\cong p\mathcal{O}_{G,F}p$. In particular, $\OGE$ and $\mathcal{O}_{G,F}$ are strongly Morita equivalent.
\end{theorem}

Hence, up to Morita equivalence, we can assume that $E$ is a countable, row-finite graph with no sources, and thus is a full groupoid $C^*$-algebra by Theorem \ref{Thm:GroupoidRowFiniteNoSources}.


\section{Characterizing properties of $\OGE$}\label{sec:properties}

Now, we are ready to extend the characterizations of various properties of $\CG$ (and thus, the simplicity of $\OGE$) obtained in \cite{EP1} to the case of triples $(G,E,\varphi)$ with $E$ countable arbitrary graph; in this sense, recall that properties like Conditions (L) and (K), or cofinality, are preserved through the desingularization process, as shown in \cite{DrinTom}. 

Since several arguments in \cite{EP1} uses the fact that $\Et(\SGE)\cong E^{\infty}$ when $E$ is a finite graph without sources, we need to prove this fact in the case of $E$ being infinite. Because of Theorem \ref{Thm:Desin} and the previous remark, we can assume without loss of generality that $E$ is row-finite without sources. So, we will keep that fact in force for the remain of the section.


\subsection{A technical issue}\label{ssec:Technical}

Suppose that $(G,E,\varphi)$ with $E$ row-finite graph without sources. Then, it is easy to see that $\Eu(\SGE)\cong E^{\infty}$. But since $E$ is infinite, the space of filters $\widehat{E}_0$ is locally compact but not compact, whence we cannot guarantee that $\Eu(\SGE)=\Et(\SGE)$; the only we know is that $\Eu(\SGE)$ is a dense subspace of $\Et(\SGE)$. \vspace{.2truecm}

Let us recall the characterization of ultrafilters given in \cite[Lemma 12.3]{Exel1}:\vspace{.2truecm}

\noindent\emph{``A filter $\xi$ is an ultrafilter if and only if for $f\in E(\SGE)$, if $f\Cap e$ for every $e\in \xi$, then $f\in \xi$''.}\vspace{.2truecm}

Also, we need to recall the characterization of tight filters given in \cite[Theorem 12.9]{Exel1}:\vspace{.2truecm}

\noindent\emph{``A filter $\xi$ is a tight filter if and only if for every $X,Y\subset E(\SGE)$ finite subsets and for every $Z\subset E(\SGE)^{X,Y}$ finite cover one has that $X\subset \xi$ and $Y\cap \xi=\emptyset$ implies that $Z\cap \xi\ne\emptyset$''.}\vspace{.2truecm}

Given any $\omega\in E^{\infty}$ and any $n\in \N$, we denote $f_{\omega_{\vert n}}=(\omega_{\vert n},(1, s(\omega_{\vert n})),\omega_{\vert n})\in E(\SGE)$, and we define $\mathcal{F}_{\omega}:=\{f_{\omega_{\vert n}} \mid n\in \N\}$. Then, the map
$$
\begin{array}{cccc}
\tau: & E^{\infty} & \rightarrow  & \Eu(\SGE)  \\
 & \omega & \mapsto  & \mathcal{F}_{\omega}  
\end{array}
$$
is a homeomorphism. In particular, if $\xi \in \Eu(\SGE) $, then $\xi$ is an infinite set of idempotents.

Now, suppose that $\xi\in \Et(\SGE)\setminus \Eu(\SGE)$. We have two options:
\begin{enumerate}
\item If $\xi$ is an infinite set, for each $n\in \N$ define $\xi_n:=\{e\in \xi \mid e=(\alpha, (1, s(\alpha)),\alpha) \text{ with  } \vert \alpha \vert =n\}$. Given $e,f\in \xi_n\subset \xi$, $0\ne ef$ because $\xi$ is a filter, and then $e=f$. Thus, for any $n\in \N$ we have that $\vert \xi_n\vert\leq 1$.

Now, for $n\leq m$ natural numbers, suppose that $\xi_n=\{(\alpha, (1, s(\alpha)), \alpha)\}$ and $\xi_m=\{(\beta, (1, s(\beta)), \beta)\}$. Again for the fact that $\xi$ is a filter, we conclude that $\alpha$ is the prefix $\beta_{\vert n}$ of $\beta$. In particular, if for some $m\in \N$ we have $\xi_m=\{(\beta, (1, s(\beta)), \beta)\}$, then for any $n\leq m$ we have $\xi_m=\{(\beta_{\vert n}, (1, s(\beta_{\vert n})), \beta_{\vert n})\}$. Thus, for any $n\in \N$ we have that $\vert \xi_n\vert= 1$.

So, we can construct $\omega \in E^{\infty}$ such that, for each $n\in \N$, $\xi_n=\{f_{\omega_{\vert n}}\}$. Hence, $\mathcal{F}_{\omega}\subseteq \xi$, and since $\mathcal{F}_{\omega}$ is maximal, we conclude that $\mathcal{F}_{\omega}= \xi$, contradicting the assumption.
\item If $\xi$ is a finite set, and $\vert \xi \vert =n$, the same argument as in case (1) states that the idempotents in $\xi$ are associated to prefixes of a fixed finite path $\alpha$. Now, take any  $X\subseteq \xi$ and any finite subset $Y\subseteq E(\SGE)\setminus \xi$. Since $E$ has no sources, be can chose a finite path $\beta$ of positive length such that $f:=(\alpha \beta, (1, s(\beta)), \alpha \beta)\not\in Y$. Thus, 
$$Z:=\{f\}\subset E(\SGE)^{X,Y}=\bigcap\limits_{x\in X}\mathcal{J}_x\cap \bigcap\limits_{y\in Y}\mathcal{J}^{\perp}_y.$$
But $f\Cap e$ for every $e\in \xi$, and since $\xi$ is a tight filter, we conclude that $f\in \xi$, contradicting the choice of $f$.
\end{enumerate}

Summarizing

\begin{proposition}\label{Prop:UltraIsTight}
If $(G,E,\varphi)$ is a triple with $E$ row-finite graph without sources, then $\Et(\SGE)= \Eu(\SGE)$.
\end{proposition}


\subsection{The properties}\label{ssec:Properties}

Finally, we are ready to obtain the desired characterizations. Notice that, since the final aim is to characterize simplicity of $\OGE$, and this property is Morita invariant, using Theorem \ref{Thm:Desin} we can reduce the problem to the case $E$ is a row-finite graph without sources. Moreover, under this restriction Proposition \ref{Prop:UltraIsTight} holds. So, we can use the arguments in \cite{EP1} involving actions $\SGE \curvearrowright E^{\infty}$ in this context. Thus, we can look at the results in \cite[Sections 12-15]{EP1} and fix the hypotheses to make work them in this context.

First, with respect to Hausdorffness of $\CG$, we have the following result.

\begin{theorem}[c.f. {\cite[Theorem 12.2]{EP1}}]\label{Thm:Hauss}
Let $(G,E,\varphi)$ a triple with $E$ being a row-finite graph without sources. Then, the following are equivalent:
\begin{enumerate}
\item For every $g\in G$ and for every $x\in E^0$ there exists a finite number of minimal strongly fixed paths for $g$ with range $x$.
\item $\CG$ is Hausdorff.
\end{enumerate}
\end{theorem}
\begin{proof}
It is the same proof as this of \cite[Theorem 12.2]{EP1}.
\end{proof}

Next, we characterize minimality, as follows.

\begin{theorem}[c.f. {\cite[Theorem 13.6]{EP1}}]\label{Thm:Minimal}
Let $(G,E,\varphi)$ a triple with $E$ being a row-finite graph without sources. Then, the following are equivalent:
\begin{enumerate}
\item The standard action $G\curvearrowright E^{\infty}$ is irreducible.
\item $\CG$ is minimal.
\item $E$ is weakly $G$-transitive.
\end{enumerate}
\end{theorem}
\begin{proof}
It is the same proof as this of \cite[Theorem 13.6]{EP1}.
\end{proof}

Finally, we characterize the groupoid being essentially principal, as follows.

\begin{theorem}[c.f. {\cite[Theorem 14.10]{EP1}}]\label{Thm:EssPrin}
Let $(G,E,\varphi)$ a triple with $E$ being a row-finite graph without sources. Then, the standard action $G\curvearrowright E^{\infty}$ is topologically free (equivalently, $\CG$ is essentially principal) if and only if the following two conditions hold:
\begin{enumerate}
\item Every $G$-circuit has an entry.
\item Given a vertex $x\in E^0$ and a group element $g\in G$, if $g$ fixes $Z(x)$ pointwise then $g$ is slack at $x$.
\end{enumerate}
\end{theorem}
\begin{proof}
It is the same proof as this of \cite[Theorem 14.10]{EP1}.
\end{proof}

Hence, we conclude the following characterization of simplicity.

\begin{theorem}[c.f. {\cite[Theorem 16.1]{EP1}}]\label{Thm:Simple}
Let $(G,E,\varphi)$ a triple with $E$ countable graph, and let $(G,F ,\widehat{\varphi})$ be its desingularization. If $G$ is amenable and $\mathcal{G}_{\text{tight}}^{(G,F)}$ is Hausdorff, then $\OGE$ is simple if and only if the following conditions are satisfied:
\begin{enumerate}
\item $F$ is weakly $G$-transitive.
\item Every $G$-circuit in $F$ has an entry.
\item Given a vertex $x\in F^0$ and a group element $g\in G$, if $g$ fixes $Z_F(x)$ poinwise then $g$ is slack at $x$.
\end{enumerate}
\end{theorem}
\begin{proof}
This is because of Theorem \ref{Thm:Desin}, Theorem \ref{Thm:Hauss}, Theorem \ref{Thm:Minimal}, Theorem \ref{Thm:EssPrin} and \cite[Theorem 5.1]{SimpleGroupoid}.
\end{proof}

Unfortunately, when $E$ is an infinite graph, \cite[Theorem 15.1]{EP1} is false, because the implication $(iv)\Rightarrow (i)$ do not work. Fortunately, there is a condition, generalizing  \cite[Theorem 15.1(iv)]{EP1}, which allows to show an analog result.

\begin{theorem}[c.f. {\cite[Theorem 16.1]{EP1}}]\label{Thm:PurInf}
Let $(G,E,\varphi)$ a triple with $E$ countable, row finite graph with no sinks. Then, the following are equivalent:
\begin{enumerate}
\item $\SGE$ is a locally contracting inverse semigroup.
\item The standard action $\theta: G\curvearrowright E^{\infty}$ is locally contracting.
\item $\CG$ is a locally contracting groupoid.
\item For every $x\in E^0$ there exists $\alpha_x\in E^ *$ with $r(\alpha_x)=x$ and there exists a $G$-circuit $(g, \gamma)$ with entries such that $s(\alpha_x)=r(\gamma)$.
\end{enumerate}
\end{theorem}
\begin{proof}

$(i)\Leftrightarrow (ii)$ By Proposition \ref{Prop:UltraIsTight} and \cite[Theorem 6.5]{EP2}.

$(ii)\Rightarrow (iii)$ It follows immediately by \cite[Proposition 6.3]{EP2}.

$(iii)\Rightarrow (iv)$ First, suppose that $(g, \gamma)$ is a $G$-circuit with no entries. Then, $\omega:=\gamma^1\gamma^2\cdots\in E^\infty$ is an isolated point in $E^\infty$. Thus, 
$$\emptyset \ne U=\{\omega\}\subset E^\infty$$
is an open set, and there are no nonempty open set $V\subseteq U$ and an open bisection $S\subseteq \CG$ such that $\overline{V}\subseteq S^{-1}S$ and $S\overline{V}S^{-1}\subsetneq V$. So, every $G$-circuit in $E$ must to have an entry.

Next, suppose that  $x\in E^0$  do not connect with any $G$-circuit. Consider the subtree $H$ of $E$ with root $x$. By assumption, it do not contain any $G$-circuit (in particular, any circuit). Now, consider $\widehat H$ the subgraph of $E$ generated by $H$ under the action of $G$ (on $E$), and notice that $\widehat H$ cannot contain any $G$-circuit. Moreover, for every $\alpha\in\widehat{H}^*$ and for every $g\in G$, we have $g\alpha\in \widehat{H}^*$. Thus, $(G, \widehat{H}, \varphi_{\vert \widehat{H}})$ is a subtriple of $(G,E,\varphi)$, and the inclusion $\iota: \widehat{H}\hookrightarrow E$ induces a natural inclusion $\widehat{\iota}:\mathcal{G}_{tight}^{(G, \widehat{H})}\hookrightarrow \CG$ of topological groupoids. Now, given any nonempty open subset $U\subseteq \CG^{(0)}$, there exists $\alpha\in E^*$ such that $\emptyset \ne Z(\alpha)\subseteq U$. Hence, the subtree rooted on $s(\alpha)$ is acyclic, and since the action of $G$ do not generate $G$-circuits there is no open bisection $S\subseteq \CG$ such that $\overline{V}\subseteq S^{-1}S$ and $S\overline{V}S^{-1}\subsetneq V$. Thus, $\mathcal{G}_{tight}^{(G, \widehat{H})}$ is not locally contracting, and then so does $\CG$, contradicting the assumption.

$(iv)\Rightarrow (i)$ Let $0\ne e\in E(\SGE)$, written $e=(\mu, (1, s(\mu)), \mu)$ for some $\mu\in E^*$. By hypothesis, there exist $\alpha\in E^*$ with $r(\alpha)=s(\mu)$ and a $G$-circuit $(g,\gamma)$ with entry $\tau$ such that $s(\alpha)=r(\gamma)$. If $\omega:= \gamma^1\gamma^2\cdots\in E^\infty$, there exists $k\in \N$ such that $s(\gamma^{k-1})=r(\tau)=r(\gamma^k)$ and $\tau\ne \gamma^k$. Define $\widehat{\gamma}:=\gamma^1\gamma^2\cdots \gamma^k$, $\beta:=\mu\alpha$ and $\widehat{\beta}:=\mu\alpha\widehat{\gamma}$. These are well-defined finite paths in $E$, and moreover $s(\alpha)=r(\widehat{\gamma})=g_{k+1}s(\widehat{\gamma})$. Hence, $s:=(\widehat{\beta}, (g_{k+1}^{-1}, s(\alpha)), \beta)\in \SGE$. Take $f_1:=(\beta, (1, s(\alpha)), \beta)$, and notice that
$$sf_1s^*=(\widehat{\beta}, (g_{k+1}^{-1}, s(\alpha)), \beta) (\beta, (1, s(\alpha)), \beta)  ({\beta}, (g_{k+1}^{-1}, s(\alpha)), \widehat{\beta})=(\widehat{\beta}, (1, s(\widehat{\gamma}), \widehat{\beta})\leq f_1 .$$
Also, $s^*s=(\beta, (1, s(\alpha)), \beta)=f_1\leq e$, so that $f_1\leq s^*s\leq e=e\cdot s^*s$. Now, define $\gamma':=\gamma^1\gamma^2\cdots \gamma^{k-1}\tau\ne \widehat{\gamma}$, and $f_0:=(\mu\alpha\gamma', (1, s(\tau)), \mu\alpha\gamma' )$. Hence, $f_0\leq f_1$ and $f_0\cdot s= (\mu\alpha\gamma', (1, s(\tau)), \mu\alpha\gamma' ) (\widehat{\beta}, (g_{k+1}^{-1}, s(\alpha)), \beta)= 0$ (because $\gamma'\ne \widehat{\gamma}$), whence $f_0sf_1=0$. Thus, $\SGE$ is locally contracting by \cite[Proposition 6.7]{EP2}.
\end{proof}

In particular, if $CG$ is minimal, effective, and $E$ contains at least one $G$-circuit, then $\CG$ is locally contracting. Hence,

\begin{corollary}[c.f. {\cite[Corollary 16.3]{EP1}}]\label{Cor:PurInf}
Let $(G,E,\varphi)$ a triple with $E$ countable graph, and let $(G,F ,\widehat{\varphi})$ be its desingularization. If $G$ is amenable and $\mathcal{G}_{\text{tight}}^{(G,F)}$ is Hausdorff, then whenever $\OGE$ is simple, it is necessarily also purely infinite (simple). 
\end{corollary}



\section*{Acknowledgments}

Part of this work was done  during a visit of the second author to the Departamento de Matem\'atica da Universidade Federal de Santa Catarina (Florian\'opolis, Brasil). This author thanks the host center for its warm hospitality.


\end{document}